\documentclass[reqno]{amsart}
\usepackage{amssymb,amsfonts,amstext,amsmath,amsthm} 
\usepackage{enumerate}
\usepackage{subfigure}
\usepackage{calligra}

\newtheorem{thm}{Theorem}[section]

\newtheorem{lem}[thm]{Lemma}
\newtheorem{prop}[thm]{Proposition}
\newtheorem{rem}[thm]{Remark}

\numberwithin{equation}{section}
\newcommand{\Om}{{\Omega}}

\newcommand{\R}{{\rm I}\!{\rm R}}

\def\n{\mathbf{n}}

\begin{document} 


\title[]{An indefinite concave-convex equation under a Neumann boundary condition II} 

\vspace{1cm}


\author{Humberto Ramos Quoirin}
\address{H. Ramos Quoirin \newline Universidad de Santiago de Chile, Casilla 307, Correo 2, Santiago, Chile}
\email{\tt humberto.ramos@usach.cl}

\author{Kenichiro Umezu}
\address{K. Umezu \newline Department of Mathematics, Faculty of Education, Ibaraki University, Mito 310-8512, Japan}
\email{\tt kenichiro.umezu.math@vc.ibaraki.ac.jp}


\thanks{The first author was supported by the FONDECYT grant 1161635}
\thanks{The second author was supported by JSPS KAKENHI Grant Number 15K04945}

\begin{abstract}
We proceed with the investigation of the problem $$-\Delta u = \lambda b(x)|u|^{q-2}u +a(x)|u|^{p-2}u \mbox{ in } \Omega, \quad \frac{\partial u}{\partial \n} = 0  \mbox{ on } \partial \Omega, \leqno{(P_\lambda)}  $$
where $\Omega$ is a bounded smooth domain in $\R^N$ ($N \geq2$), $1<q<2<p$, $\lambda \in \R$, and $a,b \in C^\alpha(\overline{\Omega})$ with $0<\alpha<1$. Dealing now with the case $b \geq 0$, $b \not \equiv 0$, we show the existence (and several properties) of a unbounded subcontinuum of nontrivial non-negative solutions of $(P_\lambda)$. Our approach is based on {\it a priori} bounds, a regularization procedure, and Whyburn's topological method. 
\end{abstract}


\maketitle



\section{Introduction and statements of main results}  \label{sec:intro}
\medskip

Let $\Omega$ be a bounded domain of $\R^N$ ($N\geq 2$) with smooth boundary $\partial \Omega$. This paper is devoted to the study of nontrivial non-negative solutions for the problem
$$ 
\begin{cases}
-\Delta u = \lambda b(x)u^{q-1} + a (x) u^{p-1}  & \mbox{in $\Omega$}, \\
\frac{\partial u}{\partial \mathbf{n}} = 0 & \mbox{on $\partial \Omega$},
\end{cases} \eqno{(P_\lambda)}  
$$
where 
\begin{itemize}
  \item $\Delta = \sum_{j=1}^N \frac{\partial^2}{\partial x_j^2}$ is the usual Laplacian in $\R^N$;

  \item $\lambda \in \mathbb{R}$;

  \item $1<q<2<p<\infty$;

  \item $a,b \in C^\alpha (\overline{\Omega})$ for some $\alpha \in (0,1)$, $a,b\not\equiv 0$, and $b\geq 0$;

  \item $\mathbf{n}$ is the unit outer normal to the boundary $\partial \Omega$.
\end{itemize}

By a {\it nonnegative (classical) solution} of $(P_\lambda)$ we mean a nonnegative function $u \in C^{2+\theta}(\overline{\Omega})$ for some $\theta \in (0,1)$ which satisfies $(P_\lambda)$ in the classical sense.  When $\lambda \geq 0$, the strong maximum principle and the boundary point lemma apply to $(P_\lambda)$, and as a consequence any non-trivial nonnegative solution of $(P_\lambda)$ is positive on $\overline{\Omega}$. In the sequel we call it a {\it positive solution} of $(P_\lambda)$. 

In this article, we proceed with the investigation of $(P_\lambda)$ made in \cite{RQU4}. We are now concerned with the case where $b\geq 0$ and we investigate the existence of a unbounded subcontinuum $\mathcal{C}_0 = \{ (\lambda, u) \}$ of nontrivial non-negative solutions of $(P_\lambda)$, bifurcating from the trivial line $\{ (\lambda, 0) \}$. Note that since $q<2$ the nonlinearity in $(P_\lambda)$ is not differentiable at $u=0$, so that we can not apply the standard local bifurcation theory \cite{CR73} directly. When $a\equiv 0$, $\Gamma_0 = \{ (0,c) : \mbox{$c$ is a positive constant} \}$ is a continuum of positive solutions of $(P_\lambda)$ bifurcating at $(0,0)$, and there is no positive solution for any $\lambda \not= 0$. Throughout this paper we shall then assume $a\not\equiv 0$, and we shall observe that the existence and behavior of $\mathcal{C}_0$ depend on the sign of $a$. 

To state our main results we introduce the following sets: 
\begin{align*}
\Omega^a_\pm = \{ x \in \Omega : a(x)\gtrless 0 \}, \quad 
\Omega^b_+ =\{ x \in \Omega : b(x) > 0 \}.  
\end{align*}
We remark that $\Omega^a_\pm$, $\Omega^b_+$ are all open subsets of $\Omega$. We shall use the following conditions on these sets:
\begin{enumerate}

\item[$(H_1)$] $\Omega^a_\pm$ are both smooth subdomains of $\Omega$, with either 
\begin{align}
& \overline{\Omega^a_+} \subset \Omega \ \ \mbox{and} \ \ \Omega = \overline{\Omega^a_+} \cup \Omega^a_-, \ \ \mbox{or} \label{H1i} \\
& \overline{\Omega^a_-} \subset \Omega \ \ \mbox{and} \ \ \Omega = \overline{\Omega^a_-} \cup \Omega^a_+.  \label{H1ii} 
\end{align}

\item[$(H_2)$] Under $(H_1)$ there exist a function $\alpha^+$ which is continuous, positive, and bounded away from zero in a tubular neighborhood of $\partial \Omega^a_+$ in $\Omega^a_+$ and $\gamma>0$ such that 
\begin{align*} 
a^+(x) = \alpha^+ (x) \, {\rm dist} (x, \partial \Omega^a_+)^{\gamma},
\end{align*}
where ${\rm dist}\, (x, A)$ denotes the distance function to a set $A$, and moreover, 
\begin{align*} 
2 < p < \min \left\{ \frac{2N}{N-2}, \frac{2N+\gamma}{N-1} \right\} \quad 
\mbox{if} \ \ N>2. 
\end{align*}
\end{enumerate}

Assumptions $(H_1)$ and $(H_2)$ are used to obtain {\it a priori} bounds on positive solutions of $(Q_{\lambda, \epsilon})$ below, cf. Amann and L\'opez-G\'omez \cite{ALG98}.
\begin{rem}{\rm 
In $(H_1)$ we may allow $\Omega^a_+ = \emptyset$ (respect.\ $\Omega^a_- = \emptyset$). In this case it is understood that $\Omega = \Omega^a_-$ (respect.\ $\Omega = \Omega^a_+$). 
}\end{rem}

Let us recall that a positive solution $u$ of $(P_\lambda)$ is said to be
{\it asymptotically stable} (respect. {\it unstable}) if $\gamma_1(\lambda,u)>0$ (respect. $<0$), where
$\gamma_1(\lambda, u)$ is the smallest eigenvalue of the linearized eigenvalue problem at $u$, namely, 
\begin{align}  \label{eigenvp:ulam}
\begin{cases}
-\Delta \phi =  \lambda (q-1) b(x) u^{q-2}\phi +(p-1)a(x)u^{p-2}\phi + \gamma \phi 
& \mbox{in $\Omega$}, \\
\frac{\partial \phi}{\partial \mathbf{n}} = 0 & \mbox{on
$\partial \Omega$}.
\end{cases}
\end{align}
In addition, $u$ is said to be {\it weakly stable} if $\gamma_1(\lambda,u)\geq 0$.

First we state a result on the existence of a unbounded subcontinuum of nontrivial non-negative solutions of $(P_\lambda)$, and its behavior and stability in the case $\int_\Omega a \geq 0$.

\begin{thm} \label{thm:ageq0}
Assume $\int_\Omega a \geq 0$, and $p\leq \frac{2N}{N-2}$ if $N>2$. Then $(P_\lambda)$ possesses a unbounded subcontinuum of 
non-negative solutions $\mathcal{C}_0 = \{ (\lambda, u) \} \subset \R \times C(\overline{\Omega})$ bifurcating at $(0,0)$. 
Moreover, the following assertions hold: 

\begin{enumerate}

\item There is no positive solution of $(P_\lambda)$ for any $\lambda \geq 0$. Consequently, if $(\lambda, u) \in \mathcal{C}_0 \setminus \{ (0,0) \}$ then $\lambda < 0$. \\ 

\item Any positive solution of $(P_\lambda)$ is unstable. \\ 

\item $\mathcal{C}_0 \cap \{ (\lambda, 0) : \lambda \not= 0 \} = \emptyset$. More precisely, for any $\Lambda > 0$ there exists $\delta_0 > 0$ such that $\max_{\overline{\Omega}} u > \delta_0$ for all nontrivial non-negative solutions of $(P_\lambda)$ with $\lambda \leq -\Lambda$. \\

\item If $(H_1)$ and $(H_2)$ hold then for any $\Lambda > 0$ there exists $C_\Lambda > 0$ such that $\max_{\overline{\Omega}}u \leq C_\Lambda$ for all $(\lambda, u) \in \mathcal{C}_0$ with $\lambda \in [-\Lambda, 0)$. Consequently, 
\begin{align*}
\{ \lambda \in \R : (\lambda, u) \in \mathcal{C}_0\setminus \{ (0,0) \} \} = (-\infty, 0). 
\end{align*}
In this case, $(P_\lambda)$ has at least one nontrivial non-negative solution for every $\lambda < 0$, see Figure \ref{fig16_0407a}.

\end{enumerate}

\end{thm}

\begin{rem}{\rm 
The non-existence result in assertion (1) of Theorem \ref{thm:ageq0} does not require the condition $p\leq \frac{2N}{N-2}$ if $N>2$. 

}\end{rem}

	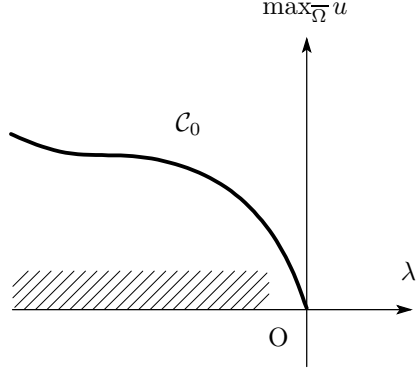
\begin{figure}[!htb]
{\unitlength 0.1in
\begin{picture}( 20.9100, 19.3800)( 16.9000,-20.8000)
%
\special{pn 8}%
\special{pa 1698 1782}%
\special{pa 3782 1782}%
\special{fp}%
\special{sh 1}%
\special{pa 3782 1782}%
\special{pa 3714 1762}%
\special{pa 3728 1782}%
\special{pa 3714 1802}%
\special{pa 3782 1782}%
\special{fp}%
%
\special{pn 8}%
\special{pa 3238 2080}%
\special{pa 3238 364}%
\special{fp}%
\special{sh 1}%
\special{pa 3238 364}%
\special{pa 3218 430}%
\special{pa 3238 416}%
\special{pa 3258 430}%
\special{pa 3238 364}%
\special{fp}%
\put(30.9000,-19.2200){\makebox(0,0){O}}%
\put(37.7000,-15.7900){\makebox(0,0){$\lambda$}}%
\put(32.2900,-2.2200){\makebox(0,0){$\max_{\overline{\Omega}}u$}}%
\put(26.2100,-8.0900){\makebox(0,0){$\mathcal{C}_0$}}%
%
\special{pn 20}%
\special{pa 3238 1774}%
\special{pa 3226 1744}%
\special{pa 3216 1714}%
\special{pa 3194 1654}%
\special{pa 3170 1594}%
\special{pa 3156 1564}%
\special{pa 3128 1506}%
\special{pa 3080 1422}%
\special{pa 3062 1394}%
\special{pa 3044 1368}%
\special{pa 3004 1316}%
\special{pa 2982 1292}%
\special{pa 2962 1268}%
\special{pa 2916 1222}%
\special{pa 2892 1200}%
\special{pa 2866 1180}%
\special{pa 2842 1160}%
\special{pa 2816 1142}%
\special{pa 2788 1124}%
\special{pa 2732 1092}%
\special{pa 2704 1078}%
\special{pa 2674 1064}%
\special{pa 2614 1040}%
\special{pa 2584 1030}%
\special{pa 2552 1020}%
\special{pa 2522 1012}%
\special{pa 2458 996}%
\special{pa 2426 990}%
\special{pa 2362 982}%
\special{pa 2328 978}%
\special{pa 2264 974}%
\special{pa 2200 974}%
\special{pa 2166 972}%
\special{pa 2134 972}%
\special{pa 2070 970}%
\special{pa 2040 968}%
\special{pa 1976 960}%
\special{pa 1946 954}%
\special{pa 1914 946}%
\special{pa 1854 928}%
\special{pa 1824 918}%
\special{pa 1734 882}%
\special{pa 1704 868}%
\special{pa 1690 862}%
\special{fp}%
%
\special{pn 4}%
\special{pa 2680 1580}%
\special{pa 2480 1780}%
\special{fp}%
\special{pa 2620 1580}%
\special{pa 2420 1780}%
\special{fp}%
\special{pa 2560 1580}%
\special{pa 2360 1780}%
\special{fp}%
\special{pa 2500 1580}%
\special{pa 2300 1780}%
\special{fp}%
\special{pa 2440 1580}%
\special{pa 2240 1780}%
\special{fp}%
\special{pa 2380 1580}%
\special{pa 2180 1780}%
\special{fp}%
\special{pa 2320 1580}%
\special{pa 2120 1780}%
\special{fp}%
\special{pa 2260 1580}%
\special{pa 2060 1780}%
\special{fp}%
\special{pa 2200 1580}%
\special{pa 2000 1780}%
\special{fp}%
\special{pa 2140 1580}%
\special{pa 1940 1780}%
\special{fp}%
\special{pa 2080 1580}%
\special{pa 1880 1780}%
\special{fp}%
\special{pa 2020 1580}%
\special{pa 1820 1780}%
\special{fp}%
\special{pa 1960 1580}%
\special{pa 1760 1780}%
\special{fp}%
\special{pa 1900 1580}%
\special{pa 1710 1770}%
\special{fp}%
\special{pa 1840 1580}%
\special{pa 1700 1720}%
\special{fp}%
\special{pa 1780 1580}%
\special{pa 1700 1660}%
\special{fp}%
\special{pa 2740 1580}%
\special{pa 2540 1780}%
\special{fp}%
\special{pa 2800 1580}%
\special{pa 2600 1780}%
\special{fp}%
\special{pa 2860 1580}%
\special{pa 2660 1780}%
\special{fp}%
\special{pa 2920 1580}%
\special{pa 2720 1780}%
\special{fp}%
\special{pa 2980 1580}%
\special{pa 2780 1780}%
\special{fp}%
\special{pa 3030 1590}%
\special{pa 2840 1780}%
\special{fp}%
\special{pa 3040 1640}%
\special{pa 2900 1780}%
\special{fp}%
\special{pa 3040 1700}%
\special{pa 2960 1780}%
\special{fp}%
\end{picture}}%
	  \caption{A unbounded subcontinuum of nontrivial non-negative solutions in the case $\int_\Omega a \geq 0$.}   
	\label{fig16_0407a}
	    \end{figure} 

To state our result corresponding to Theorem \ref{thm:ageq0} in the case $\int_\Omega a < 0$ we consider the following eigenvalue problem:
\begin{align} \label{eigenpro:chap2}
\begin{cases}
-\Delta \phi = \lambda b(x) \phi + \sigma \phi & \mbox{in $\Omega$}, \\
\frac{\partial \phi}{\partial \mathbf{n}} = 0 & \mbox{on $\partial \Omega$}. 
\end{cases}
\end{align}
For $\lambda > 0$ we denote by $\sigma_\lambda$ the smallest eigenvalue of 
\eqref{eigenpro:chap2}, which is simple and principal, and by $\phi_\lambda$ a positive eigenfunction associated with $\sigma_\lambda$. Note that $\sigma_\lambda < 0$. 

We shall deal with the following cases:
\begin{enumerate}
  \item[$(H_{01})$] $\Omega^a_+ \cap \Omega^b_+ \not= \emptyset$. 

  \item[$(H_{02})$] $\Omega^a_+ = \emptyset$. 
\end{enumerate}

\begin{thm} \label{maint} 
Assume $\int_\Omega a < 0$, and $p<\frac{2N}{N-2}$ if $N>2$. 
Then $(P_\lambda)$ possesses a unbounded subcontinuum of non-negative solutions $\mathcal{C}_0 = \{ (\lambda, u) \} \subset \R \times C(\overline{\Omega})$ bifurcating at $(0,0)$ and such that  
$\left( \mathcal{C}_0 \setminus \{ (0,0) \} \right) \cap \left([0,\infty) \times C(\overline{\Omega})\right)$ consists of positive solutions of $(P_\lambda)$. Moreover the following assertions hold:
\begin{enumerate}

\item There exists $\delta_0 > 0$ such that $\max_{\overline{\Omega}}u > \delta_0$ for all nontrivial non-negative solutions of $(P_\lambda)$ with $\lambda \leq 0$. Consequently, $\mathcal{C}_0$ bifurcates to the region $\lambda > 0$ at $(0,0)$ and does not meet $\{ (\lambda, 0): \lambda < 0 \}$. \\

\item Let $\Lambda > 0$. Then there exists $c_\Lambda > 0$ such that $u\geq c_\Lambda \phi_\Lambda$ on $\overline{\Omega}$ for all positive solutions $u$ of $(P_\lambda)$ with $\lambda \geq \Lambda$. Consequently, $\mathcal{C}_0$ does not meet $\{ (\lambda, 0) : \lambda > 0 \}$. \\

\item For some $\Lambda_0 \in (0, \infty]$, $\mathcal{C}_0$ contains $\{ (\lambda, \underline{u}_\lambda): 0<\lambda<\Lambda_0 \}$, where $\underline{u}_\lambda$ is the minimal positive solution  of $(P_\lambda)$ for $\lambda \in (0, \Lambda_0)$, i.e. $\underline{u}_\lambda \leq u$ on $\overline{\Omega}$ for all positive solutions $u$ of $(P_\lambda)$. In addition, we have:\\
\begin{enumerate}

  \item $\lambda \mapsto \underline{u}_\lambda$ is increasing;

  \item $\lambda \mapsto \underline{u}_\lambda$ is $C^\infty$ from $(0, \Lambda_0)$ to $C^{2+\alpha}(\overline{\Omega})$;

  \item $\underline{u}_\lambda \to 0$ and $\lambda^{-\frac{1}{p-q}}\underline{u}_\lambda \to c^*$ in $C^{2+\alpha}(\overline{\Omega})$ as $\lambda \to 0^+$, where $c^*=\left(\frac{\int_\Omega b}{-\int_\Omega a}\right)^{\frac{1}{p-q}}$;

  \item $\underline{u}_\lambda$ is asymptotically stable for $\lambda \in (0, \Lambda_0)$. \\

\end{enumerate} 
Finally, there exists $\delta > 0$ such that if $|\lambda| \leq \delta$ and $u$ is a positive solution of $(P_\lambda)$ such that $\max_{\overline{\Omega}}u \leq \delta$ then $(\lambda, u) \in \mathcal{C}_0$. \\

\item If $(H_{01})$ holds then 
\begin{align} \label{Lam0finite}
\Lambda_0 < \infty. 
\end{align}
Moreover, the following assertions hold:\\

\begin{enumerate}
\item $(P_\lambda)$ has a minimal positive solution $\underline{u}_{\Lambda_0}$ for $\lambda = \Lambda_0$, and $\lambda \mapsto \underline{u}_\lambda$ is continuous from $(0, \Lambda_0]$ to $C^{2+\alpha}(\overline{\Omega})$. \\

\item $\mathcal{C}_0$ consists of a smooth curve around $(\Lambda_0, \underline{u}_{\Lambda_0})$. More precisely, it is given by $(\lambda (s), u(s))$, $|s|<s_1$ (for some $s_1 > 0$)  with $\lambda (0) = \Lambda_0$, $\lambda'(0)=0>\lambda''(0)$, and   $u(0) = \underline{u}_{\Lambda_0}$. Moreover, $u(s)=\underline{u}_{\lambda(s)}$ for $s \in (-s_1, 0]$;\\

\item There is no positive solution of $(P_\lambda)$ for any $\lambda > \Lambda_0$. \\

\item The minimal positive solution $\underline{u}_{\Lambda_0}$ is weakly stable. More precisely, $\gamma_1(\Lambda_0, \underline{u}_{\Lambda_0}) = 0$. \\

\item Any positive solution $u$ of $(P_\lambda)$, except   $ \underline{u}_\lambda$ for $0<\lambda \leq \Lambda_0$, is unstable. In particular, any positive solution $u$ of $(P_\lambda)$ with $(\lambda, u) \in \mathcal{C}_0\setminus \{ (\lambda, \underline{u}_\lambda) : 0< \lambda \leq \Lambda_0 \}$ is unstable. \\

\end{enumerate}

\item If $(H_{02})$ holds then $\Lambda_0 = \infty$. Moreover, the minimal positive solution $\underline{u}_\lambda$ is the only positive solution of $(P_\lambda)$ for $\lambda > 0$. \\

\item If $(H_1)$ and $(H_2)$ hold, then for any $\Lambda > 0$ there exists $C_\Lambda > 0$ such that $\max_{\overline{\Omega}} u \leq C_{\Lambda}$ for all $(\lambda, u) \in \mathcal{C}_0$ with $\lambda \in [-\Lambda, \Lambda]$. 

\end{enumerate}
\end{thm}

\begin{rem}{\rm 
\strut
\begin{enumerate}

  \item Assertion (2), assertions (3)(a)-(d) and the uniqueness result in assertion (5) of Theorem \ref{maint} do not require the condition $p<\frac{2N}{N-2}$ if $N>2$. \\

  \item In the case $\int_\Omega a < 0$, it holds under $(H_{01})$, $(H_1)$ and $(H_2)$ that 
\begin{align*}
\{ \lambda \in \R : (\lambda, u) \in \mathcal{C}_0 \} = (-\infty, \Lambda_0].
\end{align*}
Consequently, $(P_\lambda)$ has at least one nontrivial non-negative solution for every $\lambda < 0$, at least one positive solution for $\lambda = 0, \Lambda_0$, and at least two positive solutions for every $\lambda \in (0, \Lambda_0)$, see Figure \ref{fig16_0404}. 

\end{enumerate}
}\end{rem}

	\begin{figure}[!htb]
{\unitlength 0.1in
\begin{picture}( 27.2700, 22.0200)( 16.8000,-23.6000)
%
\special{pn 8}%
\special{pa 1702 2010}%
\special{pa 4408 2010}%
\special{fp}%
\special{sh 1}%
\special{pa 4408 2010}%
\special{pa 4340 1990}%
\special{pa 4354 2010}%
\special{pa 4340 2030}%
\special{pa 4408 2010}%
\special{fp}%
%
\special{pn 8}%
\special{pa 2916 2360}%
\special{pa 2916 402}%
\special{fp}%
\special{sh 1}%
\special{pa 2916 402}%
\special{pa 2896 468}%
\special{pa 2916 454}%
\special{pa 2936 468}%
\special{pa 2916 402}%
\special{fp}%
%
\special{pn 20}%
\special{pa 2916 2010}%
\special{pa 2924 1976}%
\special{pa 2942 1912}%
\special{pa 2952 1880}%
\special{pa 2964 1850}%
\special{pa 2992 1794}%
\special{pa 3028 1746}%
\special{pa 3050 1726}%
\special{pa 3074 1706}%
\special{pa 3100 1690}%
\special{pa 3128 1674}%
\special{pa 3188 1646}%
\special{pa 3252 1622}%
\special{pa 3284 1612}%
\special{pa 3316 1600}%
\special{pa 3412 1570}%
\special{pa 3442 1558}%
\special{pa 3474 1546}%
\special{pa 3504 1534}%
\special{pa 3560 1506}%
\special{pa 3588 1490}%
\special{pa 3614 1472}%
\special{pa 3640 1452}%
\special{pa 3664 1432}%
\special{pa 3686 1408}%
\special{pa 3708 1382}%
\special{pa 3728 1354}%
\special{pa 3746 1324}%
\special{pa 3762 1294}%
\special{pa 3774 1260}%
\special{pa 3784 1228}%
\special{pa 3788 1196}%
\special{pa 3790 1166}%
\special{pa 3784 1136}%
\special{pa 3774 1110}%
\special{pa 3760 1086}%
\special{pa 3738 1064}%
\special{pa 3714 1044}%
\special{pa 3688 1026}%
\special{pa 3656 1010}%
\special{pa 3624 996}%
\special{pa 3522 960}%
\special{pa 3488 950}%
\special{pa 3422 932}%
\special{pa 3390 922}%
\special{pa 3358 914}%
\special{pa 3326 908}%
\special{pa 3296 900}%
\special{pa 3236 888}%
\special{pa 3204 882}%
\special{pa 3114 866}%
\special{pa 3054 858}%
\special{pa 2992 850}%
\special{pa 2960 848}%
\special{pa 2928 844}%
\special{pa 2898 840}%
\special{pa 2834 836}%
\special{pa 2770 830}%
\special{pa 2706 826}%
\special{pa 2672 822}%
\special{pa 2512 812}%
\special{pa 2478 810}%
\special{pa 2446 806}%
\special{pa 2382 802}%
\special{pa 2350 798}%
\special{pa 2318 796}%
\special{pa 2190 784}%
\special{pa 2158 780}%
\special{pa 2094 774}%
\special{pa 2062 770}%
\special{pa 2030 768}%
\special{pa 1966 760}%
\special{pa 1936 756}%
\special{pa 1904 752}%
\special{pa 1872 750}%
\special{pa 1776 738}%
\special{pa 1746 734}%
\special{pa 1682 726}%
\special{pa 1680 726}%
\special{fp}%
\put(27.5200,-21.6300){\makebox(0,0){O}}%
\put(43.9300,-17.7800){\makebox(0,0){$\lambda$}}%
\put(29.0900,-2.3800){\makebox(0,0){$\max_{\overline{\Omega}}u$}}%
\put(33.4900,-7.2600){\makebox(0,0){$\mathcal{C}_0$}}%
%
\special{pn 8}%
\special{pa 3796 1180}%
\special{pa 3796 2010}%
\special{dt 0.045}%
\put(37.9600,-21.7200){\makebox(0,0){$\Lambda_0$}}%
%
\special{pn 4}%
\special{pa 2750 1390}%
\special{pa 2140 2000}%
\special{fp}%
\special{pa 2690 1390}%
\special{pa 2080 2000}%
\special{fp}%
\special{pa 2630 1390}%
\special{pa 2020 2000}%
\special{fp}%
\special{pa 2570 1390}%
\special{pa 1960 2000}%
\special{fp}%
\special{pa 2510 1390}%
\special{pa 1900 2000}%
\special{fp}%
\special{pa 2450 1390}%
\special{pa 1840 2000}%
\special{fp}%
\special{pa 2390 1390}%
\special{pa 1780 2000}%
\special{fp}%
\special{pa 2330 1390}%
\special{pa 1720 2000}%
\special{fp}%
\special{pa 2270 1390}%
\special{pa 1690 1970}%
\special{fp}%
\special{pa 2210 1390}%
\special{pa 1690 1910}%
\special{fp}%
\special{pa 2150 1390}%
\special{pa 1690 1850}%
\special{fp}%
\special{pa 2090 1390}%
\special{pa 1690 1790}%
\special{fp}%
\special{pa 2030 1390}%
\special{pa 1690 1730}%
\special{fp}%
\special{pa 1970 1390}%
\special{pa 1690 1670}%
\special{fp}%
\special{pa 1910 1390}%
\special{pa 1690 1610}%
\special{fp}%
\special{pa 1850 1390}%
\special{pa 1690 1550}%
\special{fp}%
\special{pa 1790 1390}%
\special{pa 1690 1490}%
\special{fp}%
\special{pa 1730 1390}%
\special{pa 1690 1430}%
\special{fp}%
\special{pa 2810 1390}%
\special{pa 2200 2000}%
\special{fp}%
\special{pa 2870 1390}%
\special{pa 2260 2000}%
\special{fp}%
\special{pa 2910 1410}%
\special{pa 2320 2000}%
\special{fp}%
\special{pa 2910 1470}%
\special{pa 2380 2000}%
\special{fp}%
\special{pa 2910 1530}%
\special{pa 2440 2000}%
\special{fp}%
\special{pa 2910 1590}%
\special{pa 2500 2000}%
\special{fp}%
\special{pa 2910 1650}%
\special{pa 2560 2000}%
\special{fp}%
\special{pa 2910 1710}%
\special{pa 2620 2000}%
\special{fp}%
\special{pa 2910 1770}%
\special{pa 2680 2000}%
\special{fp}%
\special{pa 2910 1830}%
\special{pa 2740 2000}%
\special{fp}%
\special{pa 2910 1890}%
\special{pa 2800 2000}%
\special{fp}%
\special{pa 2910 1950}%
\special{pa 2860 2000}%
\special{fp}%
\end{picture}}%
	  \caption{A unbounded subcontinuum of nontrivial non-negative solutions 
in the case $\int_\Omega a < 0$.}   
	\label{fig16_0404}
	    \end{figure}
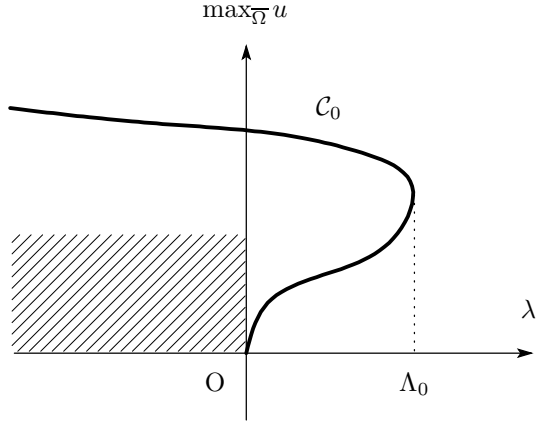 

\subsection{Notation}

Throughout this article we use the following notations and conventions: 
\begin{itemize}
\item The infimum of an empty set is assumed to be $\infty$.
\item Unless otherwise stated, for any $f \in L^1(\Om)$ the integral $\int_\Omega f$ is considered with respect to the Lebesgue measure, whereas for any $g \in L^1(\partial \Om)$ the integral $\int_{\partial \Om} g$  is considered with respect to the surface measure. 
\item For $r\geq 1$ the Lebesgue norm in $L^r (\Omega)$ will be denoted by $\| \cdot \|_r$ and the usual norm of $H^1(\Omega)$ by $\|\cdot \|$.
\item The strong and weak convergence are denoted by $\rightarrow$ and $\rightharpoonup$, respectively. 
\item The positive
and negative parts of a function $u$ are defined by $u^{\pm} :=\max
\{\pm u,0\}$. 
\item If $U \subset \R^N$ then we denote the closure of $U$ by $\overline{U}$ and the interior of $U$ by $\text{int }U$.
\item The support of a measurable function $f$ is denoted by supp $f$.
\end{itemize}

\bigskip

The rest of this article is organized as follows. In Section \ref{sec2} we prove some non-existence results. In Section \ref{sec3}, to bypass the difficulty that $(P_\lambda)$ is not differentiable at $u=0$, we consider a regularized problem with a new parameter $\epsilon > 0$ at $u=0$ and prove the existence of a unbounded subcontinuum of positive solutions for this problem. By the Whyburn topological technique we shall deduce the existence of a unbounded subcontinuum of nontrivial non-negative solutions for $(P_\lambda)$, passing to the limit as $\epsilon \to 0^+$. Section \ref{sec4} is devoted to the proofs of Theorems \ref{thm:ageq0} and \ref{maint}. 

\medskip
\section{Some non-existence results} \label{sec2} 
\medskip

First we prove the following non-existence result in the case $\int_\Omega a \geq 0$. 

\begin{prop} \label{prop:ageq0:no}
Assume $\int_\Omega a \geq 0$. Then the following two assertions hold:
\begin{enumerate}
\item There is no positive solution of $(P_\lambda)$ for any $\lambda \geq 0$. 

\item Assume $p\leq \frac{2N}{N-2}$ if $N>2$. Then, for any $\Lambda > 0$ there exists $\delta_0 > 0$ such that $\max_{\overline{\Omega}} u > \delta_0$ for all nontrivial non-negative solutions of $(P_\lambda)$ with $\lambda \leq -\Lambda$.

\end{enumerate}
\end{prop}

\begin{proof} \strut
\begin{enumerate}
\item Let $u$ be a positive solution of $(P_\lambda)$ for some $\lambda \in \R$. 
We consider two cases: 
\begin{enumerate}
\item[(i)] We assume that $a(x)\not\equiv cb(x)$ for any $c \in \R$. Then $u$ is not a constant. 
The divergence theorem provides
\begin{align*}
\int_\Omega \frac{-\Delta u}{u^{p-1}} = \int_\Omega \nabla u \nabla \left( \frac{1}{u^{p-1}}\right) 
= -\int_\Omega (p-1) |\nabla u|^2 u^{-p}< 0. 
\end{align*}
It follows that 
\begin{align*}
\int_\Omega \frac{-\Delta u}{u^{p-1}} = \int_\Omega a + \lambda \int_\Omega b u^{q-p} < 0. 
\end{align*}
Since $\int_\Omega b u^{q-p} > 0$, it should hold that $\lambda < 0$.\\ 

\item[(ii)] We assume now that $a(x) \equiv c b(x)$ for some $c \in \R$. Since $\int_\Omega a \geq 0$ and $b\geq 0$, we have $c>0$. If $u$ is a constant then it is clear that $\lambda < 0$. Otherwise we argue as in (i).\\
\end{enumerate}

\item Let $\Lambda > 0$. Assume by contradiction that there exists a sequence $(u_n)$ of nontrivial non-negative solutions  of $(P_\lambda)$ with $\lambda = \lambda_n$ such that $\lambda_n \leq -\Lambda$ and $\max_{\overline{\Omega}} u_n \to 0$ ($\lambda_n \to -\infty$ may occur). It follows that 
\begin{align} \label{ineq:410}
\int_\Omega |\nabla u_n|^2 = \int_\Omega a u_n^p + \lambda_n \int_{\Omega}b u_n^q \leq \int_\Omega a u_n^p \to 0, 
\end{align}
and consequently $ u_n  \to 0$ in $H^1(\Omega)$. We set $v_n=\frac{u_n}{\| u_n \|}$, and we assume that $v_n \rightharpoonup v_0$ for some $v_0 \in H^1(\Omega)$.  From 
\begin{align*}
\int_\Omega \nabla u_n \nabla \phi = \int_\Omega au_n^{p-1} \phi +\lambda_n \int_\Omega bu_n^{q-1} \phi, \quad \forall \phi \in H^1(\Omega), 
\end{align*}
we get $\lambda_n \int_\Omega bv_n^{q-1} \phi \to 0$ for every $\phi \in H^1(\Omega)$. It follows that $\int_\Omega bv_0^{q-1} \phi=0$ for every $\phi \in H^1(\Omega)$, so that $bv_0^{q-1} \equiv 0$. 

On the other hand, from \eqref{ineq:410} we get $\lim \int_\Omega |\nabla v_n|^2=0$, which implies $v_n \rightarrow v_0$ in $H^1(\Omega)$, and $v_0$ is a constant. Since $\| v_n \| = 1$, we have $v_0 > 0$. Hence, from $bv_0^{q-1} \equiv 0$ we obtain $b \equiv 0$, which is a contradiction.
\end{enumerate}
\end{proof}

\begin{prop}  \label{prop:a<0:no}
Assume $\int_\Omega a < 0$, and $p<\frac{2N}{N-2}$ if $N>2$. Then there exists $c_0 > 0$ such that $\max_{\overline{\Omega}} u \geq c_0$ for all nontrivial non-negative solutions $u$ of $(P_\lambda)$ with $\lambda \leq 0$. 
\end{prop}

\begin{proof}
Similarly as in the proof of Proposition \ref{prop:ageq0:no}(2), we argue by contradiction. Assume that there exists a sequence $\{ (\lambda_n, u_n) \}$ of nontrivial non-negative solutions $u_n$ of $(P_\lambda)$ with $\lambda = \lambda_n$ such that $\lambda_n \leq 0$ and $\max_{\overline{\Omega}} u_n \to 0$ ($\lambda_n \to -\infty$ may occur). It follows that $\| u_n \| \to 0$ using \eqref{ineq:410} again. Set $v_n = \frac{u_n}{\| u_n \|}$. We may assume that $v_n\rightharpoonup v_0$ for some $v_0 \in H^1(\Omega)$, and $v_n \to v_0$ in $L^p(\Omega)$. From \eqref{ineq:410} it follows that $\lim \int_\Omega |\nabla v_n|^2 = 0$. We deduce that $v_0$ is a positive constant, and $v_n \to v_0$ in $H^1(\Omega)$. On the other hand, from \eqref{ineq:410} we infer $\int_\Omega a u_n^p \geq 0$, so that $\int_\Omega a v_n^p \geq 0$. Since $v_n \to v_0$ in $L^p(\Omega)$, we have $0\leq \int_\Omega av_0^p = v_0^p \int_\Omega a$, which contradicts our assumption.
\end{proof}

\medskip
\section{Positive solutions of a regularized problem} \label{sec3}
\medskip

We consider now the existence of a subcontinuum of nontrivial non-negative solutions for $(P_\lambda)$ emanating from the trivial line. Since the mapping $t\mapsto t^{q-1}$ is not differentiable at $t=0$, we can not use the local and global bifurcation theory from simple eigenvalues \cite{CR71,CR73}. To overcome this difficulty we investigate the existence of a subcontinuum of positive solutions emanating from the trivial line for a regularized version of $(P_\lambda)$, which is formulated as 
$$
\begin{cases}
-\Delta u =  \lambda b(x) (u+\epsilon)^{q-2}u +a(x) u^{p-1} & \mbox{in $\Omega$}, \\
\frac{\partial u}{\partial \mathbf{n}} = 0  & \mbox{on $\partial \Omega$},
\end{cases} \eqno{(Q_{\lambda, \epsilon})} 
$$
where $0<\epsilon \leq 1$. Indeed, the mapping $t\mapsto (t+\epsilon)^{q-2}t$ is analytic at $t=0$. 
We remark that   $(Q_{\lambda, 0})$ corresponds to $(P_\lambda)$, so that $(P_\lambda)$ is the limiting case of $(Q_{\lambda, \epsilon})$ as $\epsilon \to 0^+$. 
To study the existence of bifurcation points on the trivial line $\{ (\lambda, 0) \}$ for $(Q_{\lambda, \epsilon})$, we consider the linearized eigenvalue problem at a nonnegative solution $u$ of $(Q_{\lambda, \epsilon})$
\begin{align} \label{ep:ep}
\begin{cases}
-\Delta \phi = a(x)(p-1)u^{p-2}\phi + \lambda b(x) \left\{ 
(q-2)(u+\epsilon)^{q-3}u + (u+\epsilon)^{q-2} \right\} \phi + \sigma \phi & \mbox{in $\Omega$}, \\
\frac{\partial \phi}{\partial \mathbf{n}} = 0 & \mbox{on $\partial \Omega$}. 
\end{cases}
\end{align}
Plugging $u=0$ into \eqref{ep:ep}, we obtain the linearized eigenvalue problem
\begin{align} \label{ep:ep0}
\begin{cases}
-\Delta \phi = \lambda \epsilon^{q-2} b(x) \phi + \sigma \phi & \mbox{in $\Omega$}, \\
\frac{\partial \phi}{\partial \mathbf{n}} =  0 & \mbox{on $\partial \Omega$}. 
\end{cases}
\end{align}
This problem has a unique principal eigenvalue $\sigma_\epsilon(\lambda)$, which is simple. Moreover we see that $\sigma_\epsilon(\lambda) > 0$ for $\lambda < 0$, $\sigma_\epsilon(\lambda) = 0$ for $\lambda = 0$, and $\sigma_\epsilon(\lambda) < 0$ for $\lambda > 0$. Note that \eqref{ep:ep0} has a positive eigenfunction associated with $\sigma_\epsilon(\lambda)$, which is a positive constant if $\lambda = 0$.

\begin{prop} \label{prop:regp:continua} 
Let $0<\epsilon \leq 1$. Then the following two assertions hold:
\begin{enumerate}

\item If $u_n$ is a positive solution of $(Q_{\lambda, \epsilon})$ for $\lambda = \lambda_n$ such that  $\max_{\overline{\Omega}} u_n \to 0$ and $\lambda_n \to \lambda^*$ for some $\lambda^* \in \R$ then $\lambda^* = 0$.

  \item $(Q_{\lambda, \epsilon})$ possesses a unbounded subcontinuum $\mathcal{C}_\epsilon = \{ (\lambda, u) \}$ in $\R \times C(\overline{\Omega})$ of positive solutions, which bifurcates at $(0,0)$ and does not meet $(\lambda, 0)$ for any $\lambda \not= 0$. 

\end{enumerate}
\end{prop}

\begin{proof}
Assertion (1) is straightforward from the fact that $\sigma_\epsilon(\lambda) > 0$ for $\lambda < 0$, and $\sigma_\epsilon(\lambda) < 0$ for $\lambda > 0$. By using assertion (1), assertion (2) is a direct consequence of the global bifurcation theory \cite{Ra71}. 
\end{proof}

\medskip
\section{Proofs of Theorems \ref{thm:ageq0} and \ref{maint}}  \label{sec4}
\medskip

\subsection{A priori upper bounds} 

The following \textit{a priori} upper bound of $\lambda$ for positive solutions of $(Q_{\lambda, \epsilon})$ follows from \cite[Proposition 6.1]{RQU4}:

\begin{prop} \label{bound:lam}
If $(H_{01})$ holds then there exists $\overline{\lambda} > 0$ such that $(Q_{\lambda, \epsilon})$ has no positive solutions for $\lambda \geq \overline{\lambda}$ and $\epsilon \in [0,1]$. 
\end{prop}

The following \textit{a priori} upper bound on the uniform norm of nonnegative 
solutions of $(Q_{\lambda, \epsilon})$ is obtained using a blow up technique from Gidas and Spruck \cite{GS81} and follows from Amann and L\'opez-G\'omez \cite{ALG98} and L\'opez-G\'omez, Molina-Meyer and Tellini \cite{LGMMT13}:

\begin{prop} \label{bound:norm}
Assume $(H_1)$ and $(H_2)$. 
Then for any $\Lambda > 0$ there exists $C_{\Lambda} > 0$ such that $\max_{\overline{\Omega}} u \leq C_{\Lambda}$ for all nonnegative solutions of $(Q_{\lambda, \epsilon})$ with $\lambda \in [-\Lambda, \Lambda]$ and $\epsilon \in [0,1]$. In particular, the conclusion holds for $(P_\lambda)$. 
\end{prop}

\begin{proof}
The case where \eqref{H1i} holds follows 
by means of Proposition \ref{app:prop:comparison} as in the proof of \cite[Proposition 6.5]{RQU4}, whereas the case where \eqref{H1ii} holds follows from the following lemma:

\begin{lem} \label{lem:bound.a<0} 
Assume $(H_1)$ with \eqref{H1ii}. Assume in addition that for any $\Lambda > 0$ there exists a constant $C_1 > 0$ such that $\max_{\overline{\Omega^a_+}}u \leq C_1$ for all 
nonnegative solutions $u$ of $(Q_{\lambda, \epsilon})$ with $\lambda \in [-\Lambda, \Lambda]$ and $\epsilon \in [0,1]$. Then, for any $\Lambda > 0$ there exists a constant $C_2$ such that $\max_{\overline{\Omega}}u \leq C_2$ for all nonnegative solutions $u$ of $(Q_{\lambda, \epsilon})$ with $\lambda \in [-\Lambda, \Lambda]$ and $\epsilon \in [0,1]$. 
\end{lem}

\begin{proof}
We use a comparison principle. For $\Lambda > 0$ we first consider the case $\lambda \in [0, \Lambda]$. Let $u$ be a 
nonnegative solution of $(Q_{\lambda, \epsilon})$. Then, since $u \leq C_1$ on $\partial \Omega^a_-$ by assumption, $u$ is a subsolution of the problem 
\begin{align} \label{pr:cc2}
\begin{cases}
-\Delta u =  \lambda b(x) (u+\epsilon)^{q-2}u -a^-(x) u^{p-1} 
& \mbox{in $\Omega^a_-$}, \\
u = C_1 & \mbox{on $\partial \Omega^a_-$}. 
\end{cases}
\end{align}
Let $w_0$ be the unique positive solution of the Dirichlet problem
\begin{align} \label{probw0}
\begin{cases}
-\Delta w = 1 & \mbox{in $\Omega^a_-$}, \\
w=0 & \mbox{on $\partial \Omega^a_-$}. 
\end{cases}
\end{align}
Set $w_1 = C(1+w_0)$ with $C>0$. Then $w_1$ is a supersolution of 
\eqref{pr:cc2} if we choose $C$ such that 
\begin{align*}
C^{2-q} = \max \left\{ C_1^{2-q}, \ \ \Lambda \left( \max_{\overline{\Omega^a_-}}b \right) \left(1 + \max_{\overline{\Omega^a_-}}w_0 \right)^{q-1} \right\}.
\end{align*}
Indeed, we observe that
\begin{align*}
-\Delta w_1 + a^- w_1^{p-1} - \lambda b (w_1+\epsilon)^{q-2}w_1 
& \geq C - \Lambda \left( \max_{\overline{\Omega^a_-}}b \right) (C(1+w_0) + \epsilon)^{q-2}C(w_0+1) \\
& \geq C - \Lambda \left( \max_{\overline{\Omega^a_-}}b \right) C^{q-1}(1+w_0)^{q-1} \\
& \geq C^{q-1} \left( C^{2-q} - \Lambda \left( \max_{\overline{\Omega^a_-}}b \right) \left( 1 + \max_{\overline{\Omega^a_-}}w_0 \right)^{q-1} \right) \geq 0. 
\end{align*}
So, the comparison principle (Proposition \ref{app:prop:comparison} in the Appendix) for \eqref{pr:cc2} yields that
\begin{align*}
u\leq C (1+w_0) \leq C \left( 1 + \max_{\overline{\Omega^a_-}}w_0 \right) \quad \mbox{on}\; \overline{\Omega^a_-}. 
\end{align*}

Next we consider the case $\lambda \in [-\Lambda, 0]$. Let $u$ be a 
non-negative solution of $(Q_{\lambda, \epsilon})$. It is straightforward that $u$ is a subsolution of the problem 
\begin{align} \label{prob:cave-}
\begin{cases}
-\Delta u = -a^-(x) u^{p-1} & \mbox{in $\Omega^a_-$}, \\
u = C_1 & \mbox{on $\partial \Omega^a_-$}. 
\end{cases}
\end{align}
Using the unique positive solution $w_0$ of \eqref{probw0}, we see that $C_1(1+w_0)$ is a supersolution of \eqref{prob:cave-}, and thus, from the comparison principle, we deduce again
\begin{align*}
u \leq C_1(1+w_0) \leq C_1 \left( 1 + \max_{\overline{\Omega^a_-}}w_0 \right) \quad \mbox{on}\; \overline{\Omega^a_-}. 
\end{align*}
Summing up, $C_2 = C \left( 1 + \max_{\overline{\Omega^a_-}}w_0 \right)$ yields the desired conclusion.
\end{proof}

The following \textit{a priori} upper bound of the uniform norm on $\overline{\Omega^a_+}$ for 
nonnegative solutions of $(Q_{\lambda, \epsilon})$ can be established in a similar manner as \cite[Theorem 6.3]{LGMMT13}.

\begin{lem} \label{lem:bound+}
Assume $(H_2)$ in addition to $(H_1)$ with \eqref{H1ii}. Then, for any $\Lambda > 0$ there exists a constant $C_1 > 0$ such that $\max_{\overline{\Omega^a_+}}u \leq C_1$ for all 
nonnegative solutions $u$ of $(Q_{\lambda, \epsilon})$ with $\lambda \in [-\Lambda, \Lambda]$ and $\epsilon \in [0,1]$. 
\end{lem}

Lemma \ref{lem:bound+} completes the proof of Proposition \ref{bound:norm} in view of Lemma \ref{lem:bound.a<0}. 
\end{proof}

\subsection{Proof of Theorem \ref{thm:ageq0}}  \label{subsec:thm1.2}
Assertions (1) and (3) follow from Proposition \ref{prop:ageq0:no}. 
By use of the Nehari manifold technique, assertion (2) can be verified in a similar way just as in \cite[Remark 2.2]{RQU4}, relying on the assumption that $\lambda < 0$, $b\geq 0$ and $b\not\equiv 0$. 

We use now a topological method proposed by Whyburn \cite{W67} to prove the existence of a unbounded subcontinuum of nontrivial non-negative solutions of $(P_\lambda)$. Let $0<\epsilon \leq 1$ and $\Lambda > 0$ be fixed. By Proposition \ref{prop:regp:continua} there exists a subcontinuum $\mathcal{C}_\epsilon'$ of positive solutions of $(Q_{\lambda, \epsilon})$ such that 
\[ 
\mathcal{C}_\epsilon' \subset \mathcal{C}_\epsilon \cap \{ (\lambda, u) \in \R \times C(\overline{\Omega}) : 
|\lambda|\leq \Lambda, \ \ \| u \|_{C(\overline{\Omega})} \leq C_\Lambda \}, 
\]
where $C_\Lambda$ is a positive constant given by Proposition \ref{bound:norm}. Then, we have $(0,0) \in \mathcal{C}_\epsilon'$, and there exists $(\lambda_\epsilon, u_\epsilon)\in \mathcal{C}_\epsilon'$ such that $ |\lambda_\epsilon| = \Lambda$. Moreover, since we can prove that $(Q_{\lambda, \epsilon})$ with $\lambda \geq 0$ and $\epsilon \in (0,1]$ has no positive solution arguing as in the proof of Proposition \ref{prop:ageq0:no}(1), we have that $\lambda < 0$ if $(\lambda, u) \in \mathcal{C}_\epsilon' \setminus \{ (0,0)\}$. Consequently, $\lambda_\epsilon = -\Lambda$, see Figure \ref{fig16_0410}.

	\begin{figure}[!htb]
{\unitlength 0.1in
\begin{picture}( 38.4500, 22.4500)( 13.3500,-27.8000)
%
\special{pn 8}%
\special{pa 1950 2440}%
\special{pa 5180 2440}%
\special{fp}%
\special{sh 1}%
\special{pa 5180 2440}%
\special{pa 5114 2420}%
\special{pa 5128 2440}%
\special{pa 5114 2460}%
\special{pa 5180 2440}%
\special{fp}%
%
\special{pn 8}%
\special{pa 3450 2780}%
\special{pa 3450 810}%
\special{fp}%
\special{sh 1}%
\special{pa 3450 810}%
\special{pa 3430 878}%
\special{pa 3450 864}%
\special{pa 3470 878}%
\special{pa 3450 810}%
\special{fp}%
%
\special{pn 8}%
\special{pn 8}%
\special{pa 2430 2440}%
\special{pa 2430 2432}%
\special{fp}%
\special{pa 2430 2395}%
\special{pa 2430 2387}%
\special{fp}%
\special{pa 2430 2350}%
\special{pa 2430 2342}%
\special{fp}%
\special{pa 2430 2305}%
\special{pa 2430 2297}%
\special{fp}%
\special{pa 2430 2259}%
\special{pa 2430 2251}%
\special{fp}%
\special{pa 2430 2214}%
\special{pa 2430 2206}%
\special{fp}%
\special{pa 2430 2169}%
\special{pa 2430 2161}%
\special{fp}%
\special{pa 2430 2124}%
\special{pa 2430 2116}%
\special{fp}%
\special{pa 2430 2079}%
\special{pa 2430 2071}%
\special{fp}%
\special{pa 2430 2034}%
\special{pa 2430 2026}%
\special{fp}%
\special{pa 2430 1989}%
\special{pa 2430 1981}%
\special{fp}%
\special{pa 2430 1944}%
\special{pa 2430 1936}%
\special{fp}%
\special{pa 2430 1898}%
\special{pa 2430 1890}%
\special{fp}%
\special{pa 2430 1853}%
\special{pa 2430 1845}%
\special{fp}%
\special{pa 2430 1808}%
\special{pa 2430 1800}%
\special{fp}%
\special{pa 2430 1763}%
\special{pa 2430 1755}%
\special{fp}%
\special{pa 2430 1718}%
\special{pa 2430 1710}%
\special{fp}%
\special{pa 2430 1673}%
\special{pa 2430 1665}%
\special{fp}%
\special{pa 2430 1628}%
\special{pa 2430 1620}%
\special{fp}%
\special{pa 2430 1582}%
\special{pa 2430 1574}%
\special{fp}%
\special{pa 2430 1537}%
\special{pa 2430 1529}%
\special{fp}%
\special{pa 2430 1492}%
\special{pa 2430 1484}%
\special{fp}%
\special{pa 2430 1447}%
\special{pa 2431 1440}%
\special{fp}%
\special{pa 2468 1440}%
\special{pa 2476 1440}%
\special{fp}%
\special{pa 2513 1440}%
\special{pa 2521 1440}%
\special{fp}%
\special{pa 2558 1440}%
\special{pa 2566 1440}%
\special{fp}%
\special{pa 2603 1440}%
\special{pa 2611 1440}%
\special{fp}%
\special{pa 2649 1440}%
\special{pa 2657 1440}%
\special{fp}%
\special{pa 2694 1440}%
\special{pa 2702 1440}%
\special{fp}%
\special{pa 2739 1440}%
\special{pa 2747 1440}%
\special{fp}%
\special{pa 2784 1440}%
\special{pa 2792 1440}%
\special{fp}%
\special{pa 2829 1440}%
\special{pa 2837 1440}%
\special{fp}%
\special{pa 2874 1440}%
\special{pa 2882 1440}%
\special{fp}%
\special{pa 2919 1440}%
\special{pa 2927 1440}%
\special{fp}%
\special{pa 2965 1440}%
\special{pa 2973 1440}%
\special{fp}%
\special{pa 3010 1440}%
\special{pa 3018 1440}%
\special{fp}%
\special{pa 3055 1440}%
\special{pa 3063 1440}%
\special{fp}%
\special{pa 3100 1440}%
\special{pa 3108 1440}%
\special{fp}%
\special{pa 3145 1440}%
\special{pa 3153 1440}%
\special{fp}%
\special{pa 3190 1440}%
\special{pa 3198 1440}%
\special{fp}%
\special{pa 3235 1440}%
\special{pa 3243 1440}%
\special{fp}%
\special{pa 3280 1440}%
\special{pa 3288 1440}%
\special{fp}%
\special{pa 3326 1440}%
\special{pa 3334 1440}%
\special{fp}%
\special{pa 3371 1440}%
\special{pa 3379 1440}%
\special{fp}%
\special{pa 3416 1440}%
\special{pa 3424 1440}%
\special{fp}%
\special{pa 3461 1440}%
\special{pa 3469 1440}%
\special{fp}%
\special{pa 3506 1440}%
\special{pa 3514 1440}%
\special{fp}%
\special{pa 3551 1440}%
\special{pa 3559 1440}%
\special{fp}%
\special{pa 3596 1440}%
\special{pa 3604 1440}%
\special{fp}%
\special{pa 3642 1440}%
\special{pa 3650 1440}%
\special{fp}%
\special{pa 3687 1440}%
\special{pa 3695 1440}%
\special{fp}%
\special{pa 3732 1440}%
\special{pa 3740 1440}%
\special{fp}%
\special{pa 3777 1440}%
\special{pa 3785 1440}%
\special{fp}%
\special{pa 3822 1440}%
\special{pa 3830 1440}%
\special{fp}%
\special{pa 3867 1440}%
\special{pa 3875 1440}%
\special{fp}%
\special{pa 3912 1440}%
\special{pa 3920 1440}%
\special{fp}%
\special{pa 3957 1440}%
\special{pa 3965 1440}%
\special{fp}%
\special{pa 4003 1440}%
\special{pa 4011 1440}%
\special{fp}%
\special{pa 4048 1440}%
\special{pa 4056 1440}%
\special{fp}%
\special{pa 4093 1440}%
\special{pa 4101 1440}%
\special{fp}%
\special{pa 4138 1440}%
\special{pa 4146 1440}%
\special{fp}%
\special{pa 4183 1440}%
\special{pa 4191 1440}%
\special{fp}%
\special{pa 4228 1440}%
\special{pa 4236 1440}%
\special{fp}%
\special{pa 4273 1440}%
\special{pa 4281 1440}%
\special{fp}%
\special{pa 4319 1440}%
\special{pa 4327 1440}%
\special{fp}%
\special{pa 4364 1440}%
\special{pa 4372 1440}%
\special{fp}%
\special{pa 4409 1440}%
\special{pa 4417 1440}%
\special{fp}%
\special{pa 4454 1440}%
\special{pa 4462 1440}%
\special{fp}%
\special{pa 4499 1440}%
\special{pa 4500 1447}%
\special{fp}%
\special{pa 4500 1484}%
\special{pa 4500 1492}%
\special{fp}%
\special{pa 4500 1529}%
\special{pa 4500 1537}%
\special{fp}%
\special{pa 4500 1574}%
\special{pa 4500 1582}%
\special{fp}%
\special{pa 4500 1620}%
\special{pa 4500 1628}%
\special{fp}%
\special{pa 4500 1665}%
\special{pa 4500 1673}%
\special{fp}%
\special{pa 4500 1710}%
\special{pa 4500 1718}%
\special{fp}%
\special{pa 4500 1755}%
\special{pa 4500 1763}%
\special{fp}%
\special{pa 4500 1800}%
\special{pa 4500 1808}%
\special{fp}%
\special{pa 4500 1845}%
\special{pa 4500 1853}%
\special{fp}%
\special{pa 4500 1890}%
\special{pa 4500 1898}%
\special{fp}%
\special{pa 4500 1936}%
\special{pa 4500 1944}%
\special{fp}%
\special{pa 4500 1981}%
\special{pa 4500 1989}%
\special{fp}%
\special{pa 4500 2026}%
\special{pa 4500 2034}%
\special{fp}%
\special{pa 4500 2071}%
\special{pa 4500 2079}%
\special{fp}%
\special{pa 4500 2116}%
\special{pa 4500 2124}%
\special{fp}%
\special{pa 4500 2161}%
\special{pa 4500 2169}%
\special{fp}%
\special{pa 4500 2206}%
\special{pa 4500 2214}%
\special{fp}%
\special{pa 4500 2251}%
\special{pa 4500 2259}%
\special{fp}%
\special{pa 4500 2297}%
\special{pa 4500 2305}%
\special{fp}%
\special{pa 4500 2342}%
\special{pa 4500 2350}%
\special{fp}%
\special{pa 4500 2387}%
\special{pa 4500 2395}%
\special{fp}%
\special{pa 4500 2432}%
\special{pa 4500 2440}%
\special{fp}%
\put(32.8000,-25.9000){\makebox(0,0){O}}%
\put(45.0000,-26.3000){\makebox(0,0){$\Lambda$}}%
\put(24.4000,-25.8000){\makebox(0,0){$-\Lambda$}}%
\put(36.8000,-12.8000){\makebox(0,0){$C_\Lambda$}}%
\put(51.4000,-22.1000){\makebox(0,0){$\lambda$}}%
\put(34.5000,-6.0000){\makebox(0,0){$\max_{\overline{\Omega}}u$}}%
\put(30.7000,-20.3000){\makebox(0,0){$\mathcal{C}_\epsilon'$}}%
\put(21.5000,-22.2000){\makebox(0,0){$(\lambda_\epsilon, u_\epsilon)$}}%
%
\special{pn 20}%
\special{pa 3450 2440}%
\special{pa 3434 2412}%
\special{pa 3398 2356}%
\special{pa 3358 2304}%
\special{pa 3338 2280}%
\special{pa 3316 2258}%
\special{pa 3292 2238}%
\special{pa 3266 2222}%
\special{pa 3238 2206}%
\special{pa 3210 2194}%
\special{pa 3178 2184}%
\special{pa 3148 2174}%
\special{pa 3114 2168}%
\special{pa 3082 2162}%
\special{pa 3050 2158}%
\special{pa 3016 2154}%
\special{pa 2984 2152}%
\special{pa 2856 2136}%
\special{pa 2824 2130}%
\special{pa 2794 2122}%
\special{pa 2764 2112}%
\special{pa 2708 2084}%
\special{pa 2680 2066}%
\special{pa 2654 2048}%
\special{pa 2626 2030}%
\special{pa 2570 1998}%
\special{pa 2540 1986}%
\special{pa 2510 1978}%
\special{pa 2480 1974}%
\special{pa 2446 1972}%
\special{pa 2430 1970}%
\special{fp}%
%
\special{pn 8}%
\special{pa 2430 1970}%
\special{pa 2358 1978}%
\special{pa 2324 1976}%
\special{pa 2294 1966}%
\special{pa 2268 1950}%
\special{pa 2250 1924}%
\special{pa 2240 1892}%
\special{pa 2244 1866}%
\special{pa 2262 1848}%
\special{pa 2290 1836}%
\special{pa 2326 1830}%
\special{pa 2366 1824}%
\special{pa 2410 1820}%
\special{pa 2454 1814}%
\special{pa 2496 1808}%
\special{pa 2534 1798}%
\special{pa 2570 1788}%
\special{pa 2602 1776}%
\special{pa 2628 1762}%
\special{pa 2646 1746}%
\special{pa 2658 1728}%
\special{pa 2662 1710}%
\special{pa 2656 1690}%
\special{pa 2644 1668}%
\special{pa 2624 1646}%
\special{pa 2600 1624}%
\special{pa 2572 1600}%
\special{pa 2540 1578}%
\special{pa 2506 1556}%
\special{pa 2470 1534}%
\special{pa 2402 1494}%
\special{pa 2368 1476}%
\special{pa 2336 1460}%
\special{pa 2274 1430}%
\special{pa 2214 1404}%
\special{pa 2186 1392}%
\special{pa 2156 1380}%
\special{pa 2130 1370}%
\special{pa 2074 1350}%
\special{pa 2048 1340}%
\special{pa 2020 1330}%
\special{pa 1990 1320}%
\special{fp}%
\put(22.0000,-11.6000){\makebox(0,0){$\mathcal{C}_\epsilon$}}%
\end{picture}}%
	  \caption{Situation of the subcontinuum $\mathcal{C}_\epsilon'$.}   
	\label{fig16_0410}
	    \end{figure}
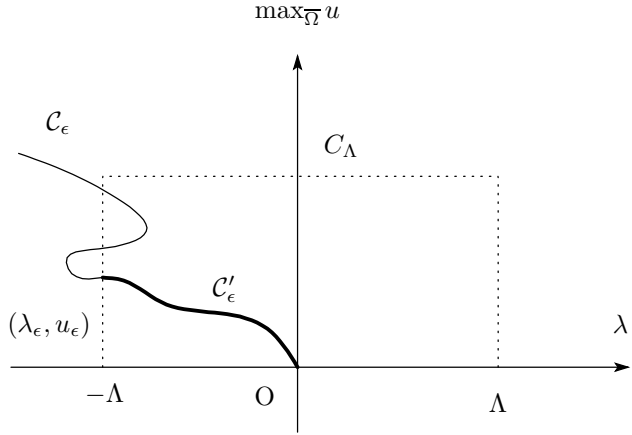 

Arguing as in Section 3 of \cite{RQU3}, we have the following facts:
\begin{itemize}

\item $\displaystyle \bigcup_{0<\epsilon \leq 1} \mathcal{C}_\epsilon'$ is precompact in $C(\overline{\Omega})$;

\item $\displaystyle{(0,0) \in \liminf_{\epsilon \to 0^+} \mathcal{C}_\epsilon'}$, i.e., it is non-empty;

\item up to a subsequence, there holds $(\lambda_\epsilon, u_\epsilon) \to (-\Lambda, u_0)$ in $\R \times C(\overline{\Omega})$, and $u_0$ is a nonnegative  solution of $(P_\lambda)$ for $\lambda = -\Lambda$. 

\end{itemize}
Hence we use 
(9.12) Theorem in page 11 of \cite{W67}, to deduce that $\mathcal{C}_0 := \limsup_{\epsilon \to 0^+}\mathcal{C}_\epsilon'$ is non-empty, closed and connected, i.e., it is a subcontinuum. Furthermore, we can check that $\mathcal{C}_0$ is contained in the set of nonnegative weak solutions of $(P_\lambda)$ (and therefore in the set of nonnegative solutions of $(P_\lambda)$, by elliptic regularity).

Finally, we shall show that $\mathcal{C}_0 \setminus \{ (0,0) \}$ consists of nontrivial non-negative solutions of $(P_\lambda)$. To this end, we prove the following lemma, see Proposition \ref{prop:ageq0:no}(2). 

\begin{lem} \label{lem:Qlam<}
Assume $p \leq \frac{2N}{N-2}$ if $N>2$. Then, for any $\Lambda > 0$, there exists $\delta_0 > 0$ such that $\max_{\overline{\Omega}} u > \delta_0$ for all positive solutions of $(Q_{\lambda, \epsilon})$ with $\lambda \leq -\Lambda$ and $\epsilon \to 0^+$. 
\end{lem}

\begin{proof}
The proof is carried out with a minor modification of that of Proposition \ref{prop:ageq0:no}(2). Assume that $u_n$ is a positive solution of $(Q_{\lambda_n, \epsilon_n})$ such that $\max_{\overline{\Omega}}u_n \to 0$, $\epsilon_n \to 0^+$, and $\lambda_n \leq -\Lambda$. As in the proof of Proposition \ref{prop:ageq0:no}(2), we deduce $u_n \to 0$ in $H^1(\Omega)$, and then, putting $v_n = \frac{u_n}{\| u_n \|}$, it follows that, up to a subsequence, $v_n \to v_0$ in $H^1(\Omega)$ for some positive constant $v_0$. 

Now, from the assumption of $u_n$, we derive 
\[
\int_\Omega au_n^{p-1} + \lambda_n \int_\Omega b (u_n + \epsilon_n)^{q-2}u_n = 0. 
\]
By multiplying the left hand side by $\| u_n \|^{-1}$, we deduce 
\[
\int_\Omega av_n^{p-1} \| u_n \|^{p-2} + \lambda_n \int_\Omega b(u_n + \epsilon_n)^{q-2}v_n = 0, 
\]
so that 
\[
0\leq \frac{1}{(\max_{\overline{\Omega}} u_n + \epsilon_n)^{2-q}} \int_\Omega bv_n \leq \int_\Omega b(u_n + \epsilon_n)^{q-2}v_n \longrightarrow 0. 
\]
It follows that 
\[
\int_\Omega b v_n \longrightarrow \int_\Omega bv_0 = 0.
\]
Since $v_0$ is a positive constant, we have $\int_\Omega b = 0$, a contradiction.\end{proof}

Now, we end the proof of Theorem \ref{thm:ageq0}. By definition, $(-\Lambda, u_0) \in \mathcal{C}_0$. From Lemma \ref{lem:Qlam<}, it follows that $u_0 \not\equiv 0$, so that $u_0$ is a nontrivial non-negative solution of $(P_\lambda)$ for $\lambda = -\Lambda$. 
%
Combining this assertion, Proposition \ref{prop:ageq0:no}, and the connectivity of $\mathcal{C}_0$, we deduce that $\mathcal{C}_0\setminus \{ (0,0)\}$ is contained in the set of nontrivial non-negative solutions of $(P_\lambda)$. Since $\Lambda$ is arbitrary, assertion (4) of this theorem follows, and now, $\mathcal{C}_0$ is the desired subcontinuum. We have finished the proof of Theorem \ref{thm:ageq0}. \qed

\subsection{Proof of Theorem \ref{maint}} 
The argument is similar. Assertion (1) follows from Proposition \ref{prop:a<0:no}, whereas Assertion \eqref{Lam0finite} follows from Proposition \ref{bound:lam}. Assertions (2) through (4), except \eqref{Lam0finite} and Assertion (4)(e), can be proved similarly as \cite[Theorem 1.1]{RQUA}. Assertion (4)(e) is verified carrying out the argument in \cite[Proposition 5.2(4)]{RQUA} for $\lambda > 0$, and the one in Assertion (2) of Theorem \ref{thm:ageq0} for $\lambda < 0$. Assertion (6) follows from Proposition \ref{bound:norm}. 

Now it remains to verify Assertion (5). To prove the uniqueness of a positive solution of $(P_\lambda)$ for $\lambda > 0$, we first reduce $(P_\lambda)$ to an equation with a nonlinear, compact and increasing mapping, as follows. If $u$ is a positive solution of $(P_\lambda)$ then, for a constant $\omega > 0$, we have
$$
u = K \left( \omega u + a(x) u^{p-1} + \lambda b(x) u^{q-1} \right) =: K F_\omega (u) \quad \mbox{in $C(\overline{\Omega})$}, 
$$
where $K : C(\overline{\Omega}) \to C^1(\overline{\Omega})$ is the compact mapping defined as the resolvent of the linear Neumann problem 
\begin{align*}
\begin{cases}
(-\Delta + \omega) u = \psi & \mbox{in $\Omega$}, \\
\frac{\partial u}{\partial \mathbf{n}} = 0 & \mbox{on $\partial \Omega$}.
\end{cases}
\end{align*}
More precisely, for any $\psi \in C^\theta (\overline{\Omega})$, $\theta \in (0,1)$, $K\psi \in C^{2+\theta}(\overline{\Omega})$ is the unique solution of the linear problem above. Moreover, $K$ is known to be \textit{strongly positive}, i.e. for $u \geq 0$ satisfying $u\not\equiv 0$ we have $Ku > 0$ on $\overline{\Omega}$ (we denote it by $Ku \gg 0$).

Next we shall observe that 
\begin{align}
& \mbox{for $C > 0$, $F_\omega (u)$ is non-decreasing in $0\leq u \leq C$ if $\omega$ is large enough}, \label{F:inc}\\ 
& \mbox{$F_\omega (\tau u) \geq \tau F_\omega (u)$ (and $\not\equiv \tau F_\omega (u))$ for $\tau \in (0,1)$ and $u\gg 0$.}  \label{F:subl}
\end{align}
We derive \eqref{F:inc} from the slope condition of $F_\omega$. Indeed, we see that if $0\leq u\leq v\leq C$ then 
\begin{align*}
\omega u + a(x) u^{p-1} - \{ \omega v + a(x) v^{p-1} \} 
= (u-v) \left\{ \omega + a(x)\frac{u^{p-1} - v^{p-1}}{u-v} \right\} \leq 0, 
\end{align*}
provided that $\omega$ is large. We derive \eqref{F:subl} by the direct computation
\begin{align*}
F_\omega (\tau u) - \tau F_\omega (u) 
= - a(x) \tau u^{p-1} (1-\tau^{p-2}) + \lambda b(x) \tau^{q-1} u^{q-1} (1 - \tau^{2-q}) 
\geq 0 \ (\mbox{and $\not\equiv 0$}). 
\end{align*}

Now we use a uniqueness argument from the proof of \cite[Theorem 24.2]{Am76f}.
Let $\lambda > 0$, $u_1$ be the minimal positive solution of $(P_\lambda)$, and $u_2$ another positive solution of $(P_\lambda)$. Then we have $u_1\leq u_2$. Assume by contradiction that $u_1 \not\equiv u_2$. Then, since $u_1 \gg 0$, there exists $\tau_0 \in (0,1)$ such that $u_1 - \tau_0 u_2 \geq 0$ but $u_1 - \tau_0 u_2 \in \partial P$, where $P = \{ u \in C(\overline{ \Omega}) : u \geq 0 \}$ denotes the positive cone of $C(\overline{\Omega})$ and $\partial P$ the boundary of $P$. Note that if $u \gg 0$ then $u$ is an interior point of $P$. Take a constant $C>0$ such that $u_1, u_2 \leq C$. Using \eqref{F:inc} and \eqref{F:subl} and the fact that $K$ is strongly positive, we deduce that
\begin{align*}
u_1 = KF_\omega (u_1) \geq KF_\omega (\tau_0 u_2) \gg \tau_0 KF_\omega (u_2) 
= \tau_0 u_2, 
\end{align*}
where $u\gg v$ means $u-v \gg 0$. Hence $u_1 - \tau_0 u_2$ is an interior point of $P$, which contradicts $u_1 - \tau_0 u_2 \in \partial P$. Consequently, $u_1\equiv u_2$, and the uniqueness holds. 

Moreover, under $(H_{02})$, the implicit function theorem is applicable at any positive solution of $(P_\lambda)$ with $\lambda > 0$. Therefore, based on assertion (1), we deduce that $\mathcal{C}_0 \setminus \{ (0,0) \} = \{ (\lambda, \underline{u}_\lambda) : 0< \lambda < \Lambda_0 \}$.

To prove $\Lambda_0 = \infty$, we establish an \textit{a priori} bound for positive solutions of $(P_\lambda)$ in a similar way as Proposition \ref{prop:ageq0:no}(2). For the sake of a contradiction we may assume $|\lambda_n| \leq \Lambda$, $\| u_n \| \to \infty$, and $u_n$ is a positive solution for $\lambda = \lambda_n$. Since 
\begin{align*}
\int_\Omega |\nabla u_n|^2 = \int_\Omega au_n^p + \lambda_n \int_\Omega b u_n^q \leq \lambda_n \int_\Omega b u_n^q, 
\end{align*}
we deduce 
$\limsup_n \int_\Omega |\nabla v_n|^2 \to 0$, where $v_n = \frac{u_n}{\| u_n \|}$. Hence we may assume that $v_n \to v_0$ for some $v_0 \in H^1(\Omega)$ and $v_0$ is a positive constant.  Also we have $v_n \to v_0$ in $L^{p-1}(\Omega)$. On the other hand, we see that
\begin{align*}
\int_\Omega \nabla u_n \nabla \phi = \int_\Omega au_n^{p-1}\phi + \lambda_n \int_\Omega bu_n^{q-1}\phi, \quad \forall \phi \in H^1(\Omega). 
\end{align*}
It follows that $\int_\Omega av_n^{p-1}\phi \to 0$, so that $\int_\Omega av_0^{p-1}\phi = 0$ for every $\phi \in H^1(\Omega)$. Hence we have $av_0^{p-1} \equiv 0$. Since $v_0$ is a positive constant, this contradicts the assumption $a\not\equiv 0$. Therefore we have proved that for any $\Lambda >0$ there exists $C_\Lambda > 0$ such that if $u$ is a positive solution of $(P_\lambda)$ with $\lambda \in [-\Lambda, \Lambda]$ then $\| u \| \leq C_\Lambda$, and thus, $\| u \|_{C(\overline{\Omega})} \leq C$ for some $C>0$ by elliptic regularity, as desired. 
By combining the \textit{a priori} bound and the use of the implicit function theorem, we verify assertion (5). 
 
The proof of Theorem  \ref{maint} is now complete. \qed \\

We conclude with the following remark on Theorems \ref{thm:ageq0} and \ref{maint}:

\begin{rem}{\rm 
Consider $(P_\lambda)$ with $q=1,2$. These cases do not correspond to a concave-convex nonlinearity but it is worthwhile discussing 
the nontrivial non-negative solutions set of $(P_\lambda)$. We may check that $(P_\lambda)$ still has a subcontinuum $\mathcal{C}_0$ of solutions such that $\mathcal{C}_0 \setminus \{ (0,0) \}$ consists of 
nontrivial non-negative solutions (with the same nature as in the case $q \in (1,2)$). 

\begin{enumerate}
\item \underline{Case $q=1$:} \ \ In this case, $\lambda b(x) u^{q-1} = \lambda b(x)$ does not depend on $u$, so that $(P_\lambda)$ no longer possesses the trivial line of solutions $\{ (\lambda, 0)\}$. However, when $\int_\Omega a < 0$, we can prove the existence of a subcontinuum $\mathcal{C}_1 = \{ (\lambda, u) \}$ of 
non-negative solutions bifurcating at $(0,0)$ to $\lambda > 0$ and such that $\mathcal{C}_1 \setminus \{ (0,0)\}$ consists of positive solutions of $(P_\lambda)$ 
when $\lambda \geq 0$. To this end, we carry out again the Whyburn topological argument developed in Subsection \ref{subsec:thm1.2}. Let $\mathcal{C}_q = \{ (\lambda, u)\}$, $q\in (1,2)$, be the unbounded subcontinuum of 
positive solutions of $(P_\lambda)$ bifurcating at $(0,0)$, as provided by Theorem \ref{maint}. 
Then, the topological argument in Subsection \ref{subsec:thm1.2} holds with $\epsilon$ replaced by $q$ for $\lambda \geq 0$. Note that $\overline{\lambda}$ given by Proposition \ref{bound:lam} and $C_\Lambda$ given by Proposition \ref{bound:norm} are determined uniformly as $q \to 1^+$. Moreover, we can check in the same way that  
assertions (1) through (6) in Theorem \ref{maint} hold true for $q=1$. Consequently, 
$\displaystyle \mathcal{C}_1 = \limsup_{q\to 1^+} \mathcal{C}_q|_{\lambda \geq 0}$ is our desired subcontinuum. \\

\item \underline{Case $q=2$:}\ \ In this case, $\lambda b(x) u^{q-1} = \lambda b(x) u$ is linear. There is a large literature on this case, with many results on the positive solutions set. Indeed, the general global bifurcation theory due to Rabinowitz provides the existence of a unbounded subcontinuum $\mathcal{C}_2 = \{ (\lambda, u) \}$ of solutions of $(P_\lambda)$ bifurcating at $(0,0)$ and such that $\mathcal{C}_2 \setminus \{ (0,0) \}$ consists of positive solutions. Furthermore, assertions (1) through (4) in Theorem \ref{thm:ageq0} and assertions (1) through (6) in Theorem \ref{maint} are verified in the same way, except the assertion $\Lambda_0 = \infty$ in Theorem \ref{maint}(5). Actually, this assertion is not true in general for $q=2$. Indeed, when $(H_{02})$ is satisfied, we know the following two results:
\begin{itemize}

\item If $a<0$ on $\overline{\Omega}$ then $\Lambda_0 = \infty$ (see Amann \cite[Theorem 25.4]{Am76f}).

\item Assume that $\{ x \in \Omega : a(x) = 0 \}\not= \emptyset$ and $b\equiv 1$. 
Assume additionally that $D_0 := \Omega \setminus \overline{\Omega^a_-}$ is a smooth subdomain of $\Omega$ bounded away from $\partial \Omega$. Consider the smallest eigenvalue $\lambda_1(D_0) > 0$ of the Dirichlet eigenvalue problem 
\begin{align*}
\begin{cases}
-\Delta \phi = \lambda \phi & \mbox{in $D_0$}, \\
\phi = 0 & \mbox{on $\partial D_0$}.
\end{cases}
\end{align*}
Then $\Lambda_0 = \lambda_1(D_0)$ and the minimal positive solution $\underline{u}_\lambda$ 
grows up to infinity in $C(\overline{\Omega})$ as $\lambda \to \lambda_1(D_0)^-$. 
Moreover, there is no positive solution of $(P_\lambda)$ for any $\lambda \geq \lambda_1(D_0)$ (see Ouyang \cite[Theorem 3]{O92}). \\ 

\end{itemize}
\end{enumerate}

On the other hand, it would be difficult to consider the limiting case $p = 2$ by the same approach as in the cases $q=1,2$, since our argument essentially uses the condition $p>2$. Indeed, we do not know whether Proposition 2.1(2) and Proposition 2.2 remain true for the case $p=2$. Thus, in the case $p=2$, one should follow another approach to study bifurcation from zero. \\ 
}\end{rem}


\medskip
\appendix 

\section{A slight variant of the comparison principle for concave problems} 
\medskip
In this Appendix we provide a variant of the comparison principle 
proved by Ambrosetti, Brezis and Cerami \cite[Lemma 3.3]{ABC} to mixed Dirichlet and Neumann nonlinear boundary conditions. We consider the general boundary value problem 
\begin{align} \label{gmp}
\begin{cases}
-\Delta u = f(x,u) & \mbox{in $D$}, \\
\frac{\partial u}{\partial \mathbf{n}} = g(x,u) & \mbox{on $\Gamma_1$}, \\ 
u=C_1 & \mbox{on $\Gamma_0$},  
\end{cases}
\end{align}
where:
\begin{itemize}
\item $D$ is a bounded domain of $\R^N$ with smooth boundary $\partial D$.
\item $\Gamma_0, \Gamma_1 \subset \partial D$ are disjoint, open, and smooth $(N-1)$ dimensional surfaces of $\partial D$.
\item $\overline{\Gamma_0}, \overline{\Gamma_1}$ are compact manifolds with $(N-2)$ dimensional closed boundary $\gamma = \overline{\Gamma_0} \cap \overline{\Gamma_1}$ such that $\partial D = \Gamma_0 \cup \gamma \cup \Gamma_1$.
\item $f: \overline{\Omega}\times [0,\infty) \to \R$ and $g: \overline{\Gamma_1} \times [0,\infty) \to \R$ are continuous. 
\item $C_1$ is a non-negative constant. \\
\end{itemize} 

The result \cite[Proposition A.1]{RQU2} can be slightly relaxed as follows:

\begin{prop} \label{app:prop:comparison}
Under the above conditions, assume that for every $x \in D$, $t \mapsto \frac{f(x,t)}{t}$ is decreasing in $(0,\infty)$, and for every $x \in \Gamma_1$, $t \mapsto \frac{g(x,t)}{t}$ is non-increasing in $(0,\infty)$. Let $u,v \in H^1 (D) \cap C(\overline{D})$ be non-negative functions satisfying $u\leq C_1 \leq v$ on $\Gamma_0$, and 
\begin{align} \label{ineq:sub}
\int_D \nabla u \nabla \varphi - \int_D f(x,u) \varphi - \int_{\Gamma_1} g(x,u) \varphi \leq \ 0, \quad \forall \varphi \in H_{\Gamma_0}^1(D) \ \ \mbox{such that} \ \ \varphi \geq 0, 
\end{align}
\begin{align}\label{ineq:super}
\int_D \nabla v \nabla \varphi - \int_D f(x,v) \varphi - \int_{\Gamma_1} g(x,v) \varphi \geq \ 0, \quad \forall \varphi \in H_{\Gamma_0}^1(D) \ \ \mbox{such that} \ \ \varphi \geq 0.  
\end{align}
If $v > 0$ in $D$, then $u\leq v$ on $\overline{D}$. 
\end{prop}

\begin{rem}{\rm 
\strut
\begin{enumerate}

\item In \cite[Proposition A.1]{RQU2} the case $C_1 = 0$ has been considered. 

\item Assume additionally that $f, g$ are smooth enough. 
If a non-negative function $u \in C^2(\overline{\Omega})$ satisfies
\begin{align*}
\begin{cases}
-\Delta u \leq f(x,u) & \mbox{in $D$}, \\
u\leq C_1 & \mbox{on $\Gamma_0$}, \\
\frac{\partial u}{\partial \mathbf{n}} \leq g(x,u) & \mbox{on $\Gamma_1$}, 
\end{cases}
\end{align*}
then $u$ satisfies \eqref{ineq:sub}. Similarly if the opposite inequalities hold then $u$ satisfies \eqref{ineq:super}. 

\item $\Gamma_0 = \emptyset$ (or alternatively $\Gamma_1 = \emptyset$) is allowed. 
\end{enumerate}
}\end{rem}

\begin{proof} 
Let $\theta: \R \rightarrow \R$,  be a nonnegative nondecreasing smooth function such that $\theta (t)=0$ for $t\leq 0$ and $\theta (t) = 1$ for $t\geq 1$. For $\varepsilon > 0$ we set 
$\theta_\varepsilon (t) = \theta (t/\varepsilon)$. Since $u-v \leq 0$ on $\Gamma_0$, we have $v \theta_\varepsilon (u-v) \in H_{\Gamma_0}^1(D)$, so that 
\begin{align} \label{u:ineq}
\int_D \nabla u \nabla (v \theta_\varepsilon (u-v)) - \int_D f(x,u)v \theta_\varepsilon (u-v) - \int_{\Gamma_1} g(x,u) v \theta_\varepsilon (u-v) \leq 0. 
\end{align}
Likewise, since $u \theta_\varepsilon (u-v) \in H_{\Gamma_0}^1(D)$, 
we have 
\begin{align} \label{v:ineq}
\int_D \nabla v \nabla (u \theta_\varepsilon (u-v)) - \int_D f(x,v)u \theta_\varepsilon (u-v) - \int_{\Gamma_1} g(x,v) u \theta_\varepsilon (u-v) \geq 0. 
\end{align}
Let $\Gamma_1^+ = \{ x \in \Gamma_1 : u, v > 0 \}$, and 
$D^+ = \{ x \in D : u > 0 \}$. 
Since $t \mapsto \frac{g(x,t)}{t}$ is non-increasing in $(0,\infty)$, we have $g(x,0)\geq 0$, which combined with \eqref{u:ineq} and \eqref{v:ineq} 
yields
\begin{align*}
& \int_D  u \theta_\varepsilon^\prime (u-v) \nabla v (\nabla u - \nabla v) 
- \int_D  v \theta_\varepsilon^\prime (u-v) \nabla u (\nabla u - \nabla v) 
\\
& \geq \int_{D^+} uv \left( \frac{f(x,v)}{v} - \frac{f(x,u)}{u} \right) 
\theta_\varepsilon (u-v) 
+ \int_{\Gamma_1^+} uv \left( \frac{g(x,v)}{v} - \frac{g(x,u)}{u} \right) \theta_\varepsilon (u-v) \\
& \geq \int_{D^+} uv \left( \frac{f(x,v)}{v} - \frac{f(x,u)}{u} \right) 
\theta_\varepsilon (u-v). 
\end{align*}
From $-\int_D u \theta_\varepsilon^\prime (u-v) |\nabla (u-v)|^2 \leq 0$, it follows that  
\begin{align} \label{ineq140826}
\int_D  (u-v) \theta_\varepsilon^\prime (u-v) \nabla u\nabla (u-v) 
\geq \int_{D^+} uv \left( \frac{f(x,v)}{v} - \frac{f(x,u)}{u} \right) 
\theta_\varepsilon (u-v). 
\end{align}

Now, we introduce $\gamma_\varepsilon (t) = \int_0^t s \theta_\varepsilon^\prime (s) ds$ for $t\in \R$. We have then $0\leq \gamma_\varepsilon (t) \leq \varepsilon$, $t \in \R$. Note that $\nabla (\gamma_\varepsilon (u-v)) = (u-v) \theta_\varepsilon^\prime (u-v) \nabla (u-v)$. Hence, from \eqref{ineq140826} 
we deduce that 
\begin{align*}
\int_D \nabla u \nabla (\gamma_\varepsilon (u-v)) 
\geq \int_{D^+} uv \left( \frac{f(x,v)}{v} - \frac{f(x,u)}{u} \right) 
\theta_\varepsilon (u-v). 
\end{align*}
Now, since $\gamma_\varepsilon (u-v) \in H_{\Gamma_0}^1(D)$ and 
$\gamma_\varepsilon (u-v) \geq 0$, we note that
\begin{align*}
\int_D \nabla u \nabla (\gamma_\varepsilon (u-v)) 
- \int_D f(x,u) \gamma_\varepsilon (u-v) 
- \int_{\Gamma_1} g(x,u) \gamma_\varepsilon (u-v) \leq 0, 
\end{align*}
and combining the two latter assertions, we get
\begin{align*}
\int_D f(x,u) \gamma_\varepsilon (u-v) + \int_{\Gamma_1} g(x,u) \gamma_\varepsilon (u-v) \geq \int_{D^+} uv \left( \frac{f(x,v)}{v} - \frac{f(x,u)}{u} \right) \theta_\varepsilon (u-v). 
\end{align*}
Since $\gamma_\varepsilon (t)\leq \varepsilon$, there exists a constant $C>0$ such that 
\begin{align} \label{ineq:Cvarep}
C\varepsilon \geq \int_{D^+} uv \left( \frac{f(x,v)}{v} - \frac{f(x,u)}{u} \right) \theta_\varepsilon (u-v). 
\end{align}
Since $t \mapsto \frac{f(x,t)}{t}$ is decreasing in $(0,\infty)$, 
we use Fatou's lemma to deduce from \eqref{ineq:Cvarep} that  
\begin{align*}
\int_{D^+} \liminf_{\varepsilon \to 0^+}\, uv \left( \frac{f(x,v)}{v} - \frac{f(x,u)}{u} \right) \theta_\varepsilon (u-v) \leq 0. 
\end{align*}
Note that
\begin{align*}
\lim_{\varepsilon \to 0^+} \theta_\varepsilon (u-v) = \left\{ 
\begin{array}{ll}
1, & u > v, \\
0, & u\leq v, 
\end{array} \right. 
\end{align*}
so that 
\begin{align*}
\int_{D^+ \cap \{ u>v \}} uv \left( \frac{f(x,v)}{v} - \frac{f(x,u)}{u} \right) \leq 0. 
\end{align*}
Using again that $t \mapsto \frac{f(x,t)}{t}$ is decreasing in $(0,\infty)$, we conclude 
that $|D^+ \cap \{ u>v \}|=0$, and since $u \equiv 0<v$ in $D \setminus D^+$, we have $u\leq v$ a.e.\ in $D$. By continuity, the desired conclusion follows. 
\end{proof}



\begin{thebibliography}{1000}

\bibitem{Am76f} H.\ Amann, Fixed point equations and nonlinear eigenvalue problems in ordered Banach spaces, SIAM Rev.\ {\bf 18}, (1976), 620--709. 




\bibitem{ALG98} H.\ Amann and J.\ L\'opez-G\'omez, A priori bounds and multiple
solutions for superlinear indefinite elliptic problems, J.\ Differential Equations {\bf 146}, (1998), 336--374. 


\bibitem{ABC} A.\ Ambrosetti, H.\ Brezis, and G.\ Cerami, Combined effects of concave and convex nonlinearities in some elliptic problems, J.\ Funct.\ Anal.\ {\bf 122}, (1994), 519--543. 










\bibitem{CR71} M.\ G.\ Crandall and P.\ H.\ Rabinowitz, Bifurcation from simple eigenvalues, J.\ Functional Analysis {\bf 8}, (1971), 321--340. 


\bibitem{CR73} M.\ G.\ Crandall and P.\ H.\ Rabinowitz, Bifurcation, perturbation of simple eigenvalues and linearized stability, Arch.\ Rational Mech.\ Anal.\ {\bf 52}, (1973), 161--180. 














\bibitem{GS81} B.\ Gidas and J.\ Spruck, Global and local behavior of positive solutions of nonlinear elliptic equations, Comm.\ Pure Appl.\ Math. \textbf{34}, (1981), 525--598.




\bibitem{LGMMT13} J.\ L\'opez-G\'omez, M.\ Molina-Meyer and A.\ Tellini, The uniqueness of the linearly stable positive solution for a class of superlinear indefinite problems with nonhomogeneous boundary conditions, J.\ Differential Equations {\bf 255}, (2013), 503--523. 





\bibitem{O92}T.\ Ouyang, On the positive solutions of semilinear equations $\Delta u+\lambda u-hu\sp p=0$ on the compact manifolds, Trans.\ Amer.\ Math.\ Soc.\ {\bf 331},  (1992), 503--527. 




\bibitem{Ra71} P.\ H.\ Rabinowitz, Some global results for nonlinear eigenvalue problems, J.\ Functional Analysis \textbf{7}, (1971), 487--513. 




\bibitem{RQU2} H.\ Ramos Quoirin and K.\ Umezu, Positive steady states of an indefinite equation with a nonlinear boundary condition: existence, multiplicity and asymptotic profiles, Calc.\ Var.\ Partial Differential Equations \textbf{55}, (2016), no.\ 4, Paper No.\ 102, 47 pp. 


\bibitem{RQU3} H.\ Ramos Quoirin and K.\ Umezu, Bifurcation for a logistic elliptic equation with nonlinear boundary conditions: A limiting case,  J.\ Math.\ Anal.\ Appl.\ {\bf 428}, (2015), 1265--1285. 


\bibitem{RQUA} H.\ Ramos Quoirin and K.\ Umezu, On a concave-convex elliptic problem with a nonlinear boundary condition, Ann.\ Mat.\ Pura Appl.\ (4), 195 (2016), no. 6, 1833--1863 . 


\bibitem{RQU4} H.\ Ramos Quoirin and K.\ Umezu, An indefinite concave-convex equation under a Neumann boundary condition I, Israel J.\ Math. (in press). 
arXiv:1603.04940









\bibitem{W67} G.\ T.\ Whyburn, Topological analysis, Second, revised edition, Princeton Mathematical Series, No.\ 23, Princeton University Press, Princeton, N.J., 1964. 






\end{thebibliography}
\end{document}